\documentclass[namedate,webpdf,imamat]{ima-authoring-template}

\graphicspath{{Fig/}}

\theoremstyle{thmstyletwo}
\newtheorem{theorem}{Theorem}[section]

\newtheorem{example}{Example}%
\newtheorem{remark}{Remark}%

\newtheorem{definition}{Definition}

\numberwithin{equation}{section}

\begin{document}

\DOI{XXXX}
\copyrightyear{XXXX}
\firstpage{1}


\title[Normalisation of Factorisation]{Normalisation of Wiener--Hopf Factorisation for Matrix Functions}

\author{Adukova N. V.*
\address{\orgdiv{Department of Mathematics}, \orgname{Aberystwyth University}, \orgaddress{{Penglais}, \postcode{SY23 3FL}}, \state{Ceredigion}, \country{UK}}}
\author{Adukov V. M.\ORCID{0000-0002-3594-3565}
\address{\orgdiv{Institute of Natural Sciences and Mathematics}, \orgname{South Ural State University}, \orgaddress{\street{Lenin prospekt}, \postcode{454080},  \country{Russia}}}}

\authormark{Adukova N. et al.}

\corresp[*]{Corresponding author: \href{email:email-id.com}{email:naa51@aber.ac.uk}}



\abstract{For a matrix function $A(t)$, we study the normalisation of its Wiener--Hopf factorisation $A(t)=A_-(t)D(t)A_+(t)$  which guarantees the uniqueness of the factors    $A_\pm(t)$. Previously, this problem was investigated for second-order matrix functions, the stable factorisation, and the factorisation of 
matrix functions with distinct partial indices. The main tools for solving the problem of normalisation are the theorem of I.C. Gohberg and M.G. Krein on a general form of the factors $A_\pm$, a Birkhoff factorisation of matrix functions, and a block PLU factorisation of the numerical matrix $A_-(\infty)$.
To solve the normalisation problem, a concept of P-normalised factorisation is introduced. The Birkhoff factorisation of matrix functions plays a key role in defining this concept. The permutation matrix $P$ is determined  by PLU factorisation of the numerical matrix $A_-(\infty)$.  We proved that any invertible matrix function $A(t)$ admits  P-normalisation. This normalisation guarantees the uniqueness of the Wiener--Hopf factorisation. The P-normalised factorisation of $A(t)$ allows us to find the Birkhoff factorisation of $A(t)$. Without normalisation, it is impossible to obtain explicit estimates of the absolute errors in the approximate computation of the factorisation factors.}


\maketitle

\section*{Introduction}

  The Wiener--Hopf technique, or, more broadly, Wiener--Hopf factorisation of matrix functions is widely used in applied mathematics.  
 Two outstanding works \citep{Plemelj1908} and \citep{WH31} form the foundation of 
this technique.

The first work was  devoted to Hilbert's 21st problem (or Riemann--Hilbert Problem in the analytic theory of linear differential equations \citep{Plemelj64, AB94}). 

In 1857, B. Riemann formulated a problem on linear differential equations with rational coefficients. In 1900, D. Hilbert included this problem as number 21 in his famous list of unsolved mathematical problems \citep{Problems1900}. In connection with this problem B. Riemann first formulated the following homogeneous conjugation problem:

\emph{For the given closed  contour $\Gamma$ and a piecewise constant non-singular matrix function $G(t)$ defined on $\Gamma$ it is necessary to find a vector function $\Phi(z)$ that is analytic
on $\mathbb C \setminus \Gamma$ such that its limit values satisfy on $\Gamma$ the conjugation condition
\begin{equation}\label{conjug}
\Phi^+(t)=G(t)\Phi^-(t), \ t\in \Gamma.
\end{equation} 
}

B. Riemann  did not formulate precisely the conditions on $\Phi(z)$ and did not try to solve the problem. The first solution of the problem by reducing it to the Fredholm integral equation under stronger conditions on $\Gamma$ and $G(t)$ was given by D.~Hilbert. However, as N. I.~Muskhelishvili writes in his monograph 
\citep{Muskhelishvili11}, ``the solution given by D. Hilbert is very complicated and cannot be considered complete''. Nevertheless, the conjugation problem (\ref{conjug}) is now called \emph{the Riemann-Hilbert boundary value problem}.

The complete study of the homogeneous Riemann--Hilbert boundary value problem was carry out by J. Plemelj \citep{Plemelj1908}. He used Cauchy-type integrals and assumed that entries of $G(t)$ are H\"older functions. J. Plemelj obtained the general solution
of (\ref{conjug}) in terms of a so-called canonical system of solutions of (\ref{conjug}). 
Further research on this problem followed the path outlined by J. Plemely. 
At present, the theory of (\ref{conjug}) concerning the existence of a solution has been rigorously and comprehensively developed thanks to the works of N.~I.~Muskhelishvili and N.~P.~Vekua \citep{Muskhelishvili11, Vekua67}.  In the scalar case an elementary and quite effective solution based on Plemelj's ideas was obtained by F. D.~Gakhov \citep{Gakhov66} earlier than in the matrix case.


We now indicate how the Riemann--Hilbert boundary value problem is related to
the Wiener--Hopf factorisation of matrix functions. Let $\Gamma$ be a Jordan smooth
contour and $G(t)$ be a H\'older $n\times n$ non-singular matrix function on it. In the theory  
of (\ref{conjug}), it is proved (see \cite{Muskhelishvili11, Vekua67}) that there is a system of solutions $\Phi_1(z),\ldots,\Phi_n(z)$ such that 
$\det \left(\Phi_1(z),\ldots,\Phi_n(z)\right) \ne 0$ at any finite point of the complex plane $\mathbb C$. At infinity every solution $\Phi_j(z)$ has  finite order $k_j\in \mathbb Z$. Moreover, 
$\det \left(z^{k_1}\Phi_1(z),\ldots,z^{k_n}\Phi_n(z)\right)\ne 0, \infty$ at infinity.
This system is called \emph{a canonical system} for (\ref{conjug}), and integers
$k_1, \ldots, k_n$ are called \emph{the partial indices } of $G(t)$.
Now the conjugation condition (\ref{conjug}) gives
$$
\left(\Phi^+_1(t),\ldots,\Phi_n^+ (t)\right)=
G(t)\left(t^{k_1}\Phi_1^-(t),\ldots,t^{k_n}\Phi_n^-(t)\right)
{\rm diag}[t^{-k_1},\ldots,t^{-k_n}]     , \ t\in \Gamma.
$$
Denote
$$
G_+(t)=\left(\Phi^+_1(t),\ldots,\Phi_n^+ (t)\right), \quad 
G_-(t)=\left(t^{k_1}\Phi_1^-(t),\ldots,t^{k_n}\Phi_n^-(t)\right), 
$$
and $D(t)= {\rm diag}[t^{k_1},\ldots,t^{k_n}]$. The factors $G_\pm(t)$
have all properties necessary for factors from a Wiener--Hopf factorisation.
The precise definition of the Wiener--Hopf factorisation will be given in the next section.

We obtain the left Wiener--Hopf factorisation of $G(t)$ in the form
$$
G(t)=G_+(t)D(t)G_-^{-1}(t).
$$

Conversely, if we know the factorisation of $G(t)$, we can construct a canonical system of the conjugate problem \eqref{conjug}. Thus, the Riemann-Hilbert boundary value problem is equivalent to the Wiener--Hopf factorisation problem.

In 1931 the mathematical community  was amazed by the appearance of the paper of Norbert Wiener and Eberhard Hopf \citep{WH31}. This article became the basis of the Wiener--Hopf factorisation technique. The authors proposed an explicit method for solving an integral equation on a half-line with a kernel depending on the difference of the arguments. Before this article appeared, it was believed that exact analytical methods for solving such equations could not exist in principle. N. Wiener and E. Hopf 
unexpectedly refuted this belief. They applied the Fourier integral transform to their integral equation. As a result, they obtained a linear functional equation for two unknown functions, which must be carried out in some strip. It is assumed that one of the unknown functions is
analytic in the half-plane above the strip, and the other is below. To solve this functional equation they proposed to factorise the ratio of the coefficient of the given functional equation. This simple and elegant approach allowed them to solve the original integral equation. The key point of the method is the functional equation, now called the Wiener--Hopf functional equation. The technique based on this functional equation is called the Wiener--Hopf technique. 
Details of this technique can be found in \citep{Noble88}. See also the brief but very informative review \citep{LA07}. An essential part of the Wiener--Hopf technique is
the Wiener--Hopf factorisation of functions or matrix functions in the strip. 

 The appearance of the strip is caused by the conditions imposed on the kernel function of the integral equation. In 1941, E. Reissner \citep{Reissner41} weakened these conditions,
and in 1948 I. M. Rapoport \citep{Rapoport48} established a relationship between  
the Wiener--Hopf functional equation and the Riemann-Hilbert boundary value problem
(see also a modern work \citep{Kisil15a} on this subject). The final natural 
conditions, under which this class of integral equations or systems of equations is considered, were established by M. G. Krein and I. C. Gohberg in the fundamental works 
\citep{Krein62, GK60}. After these works, it became clear that the directions generated by works \citep{Plemelj1908} and \citep{WH31} are based on the same Riemann--Hilbert boundary value problem. In the first case, the problem is considered in the class of H\"older matrix functions, and in the second case  Wiener's class of matrix functions is used on the real line or the unit circle.

The further development of the theory is associated mainly with M. G. Krein and I. C. Gohberg \citep{Krein62, GK60}. They were the first to understand  that all concepts related to the canonical system are more conveniently formulate in terms of some factorisation of a matrix function. This factorisation is a key point in the Wiener--Hopf technique. For  this reason, the novel factorisation became known as the Wiener--Hopf factorisation of matrix functions.

The initial stage in the development of the factorisation theory was related to questions of the existence of factorisation, its stability, and the study of its general properties. I. C. Gohberg with co-workers established the necessary and sufficient conditions for factorisability of any matrix function with entries from a Banach algebra 
\citep{Gohberg64}. It has been established that partial indices of $A(t)$ are uniquely determined by the matrix function $A(t)$ \citetext{\citealp[Theorem~7.1]{GK60}; \citealp[Theorem~1.1]{GF74}}. In contrast to the partial indices,  the factors are not uniquely determined by $A(t)$. The general form of the factorisation factors is described by the Gohberg--Krein theorem \citetext{\citealp[Theorem~7.2]{GK60}; \citealp[Ch.~VIII, Theorem~1.2]{GF74}}. The stability of the factorisation with respect to small perturbations of the original matrix functions was also  studied. It turned out that the problem is unstable in the general case \citetext{\citealp[Theorem~10.2]{GK60}; \citealp{Bojarsky58b}}.  However, in some cases, a stability analysis of the factorisation problem can be performed \citep{Kisil15, MR16, MR18,AMR20,AA20}. 

Further, the focus of research shifted towards expanding the classes of factorisable functions and matrix functions \citep{Simonenko61, Simonenko64, LS87, BKS02}, 
and towards constructive methods in the factorisation theory. A review of available explicit and constructive factorisation methods can be found in 
\citep{Kisil13,GM13,RM16, ES20, KAMR21}. 

The Wiener--Hopf factorisation plays an important role in the study of many applied problems. It is difficult to provide  a full list of its applications. The problem arises in the soliton theory 
\citep{AS81, FT07, AM23, Zhou89}, the diffraction theory \citep{DL20}, the elasticity 
\citep{Abrahams02, AHKMNRS21, KAMR21}, the acoustics \citep{MAK26, ShM20}, the wave propagation and electromagnetics \citep{DZ14, DL20}.

Recently, considerable attention has also been devoted to diffraction problems in quarter-plane \citep{AS19, AA21, KAK25}. A distinctive feature of this research direction is the absence of an analogue of the Wiener--Hopf technique. At present, it remains unclear which method could serve a similar role in this setting.

Therefore, by the 1960s, the general theory of factorisation was basically completed.
Let us summarise  what has been achieved in the general theory of factorisation.
Necessary and sufficient conditions were found for any matrix function from a given Banach algebra to admit the standard Wiener--Hopf factorisation. It was established that factorisation factors are not uniquely determined and their general form was obtained. A necessary and sufficient condition for the stability of the partial indices has been found.

However, in the general factorisation theory the question of a normalisation of the factorisation remains unexplored. The general form of factorisation factors does not single out a unique factorisation. As a result, different constructive methods may produce different factorisation factors for the same matrix function and its non-uniqueness leads makes it impossible to compare  different constructive methods directly through their factorisation factors. Moreover, the lack of normalisation does not allow us to obtain explicit error estimates for approximate factorisations.

The purpose of the present paper is to fill this gap in the general factorisation theory.
We will introduce the concept of P-normalisation of the factorisation, which guarantees its uniqueness, and prove the existence of such a normalisation for any matrix function that admits the standard factorisation.
Such a normalisation makes it possible to compare different constructive methods, study the stability of the factorisation factors, and obtain explicit error estimates for approximate factorisations.

The paper is organised as follows. Section \ref{sec1} provides the preliminary background, including definitions and basic information on some factorisations of matrix functions and matrices that will be used for normalisation. We will need the Wiener--Hopf and Birkhoff factorisations of matrix functions and the block LU and PLU factorisations of numerical matrices. This section recalls the necessary facts concerning these factorisations.

In Section \ref{sec2}, we introduce the main concept of this paper, namely, the P-normalisation of the Wiener--Hopf factorisation $A(t)=A_-(t)D(t)A_+(t)$. Here P is a permutation matrix determined by the block LU factorisation of the matrix 
$A_-(\infty)$. 
The concept of P-normalisation was first defined in the talk by the authors at the 13th ISAAC Congress \citep{AA21b}.

Note that the permutation P is not uniquely defined since the matrix $A_-(\infty)$ can be reduced in various ways to the desired block LU factorisation. However, if we fix the choice  of
P then the P-normalised factorisation will be unique. This fact is proved in Section \ref{sec3}.

The existence of the P-normalised factorisation for any matrix function that admits
the Wiener--Hopf factorisation is proved in Section \ref{sec4}. This was previously done in some special cases: for second-order matrix functions \citep{Adukov22a}, for matrix functions with stable systems of partial indices \citep{Adukova22}, and for matrix functions with distinct partial indices \citep{Adukov22b}.

Conclusions and perspectives are discussed in Section \ref{sec5}.

\section{Preliminaries on some matrix factorisations}\label{sec1}
The main object of study in this work is the Wiener--Hopf factorisation of matrix functions. Our aim is to fill a gap in the general theory of factorisation, namely, to construct a normalisation of the factorisation that guarantees its uniqueness.
To carry out the normalisation, we must know a Wiener--Hopf factorisation of the given matrix function. The theorem of Gohberg--Krein on the general form of factorisation factors is one of the main tools in the normalisation process.  Our definition of normalisation is based on the Birkhoff factorisation of matrix functions and the block PLU factorisation of a numerical matrix.  

\subsection{Wiener--Hopf factorisation of matrix functions}
Let $\Gamma $ be a simple smooth closed contour in the complex plane $\mathbb C$ bounding the domain $D_{+}$. The complement of  $D_{+}\cup \Gamma $ in 
$\overline {\mathbb C} ={\mathbb C}\cup \{\infty\}$ will be denoted by $D_{-}$. We can assume that $0\in D_{+}$. Denote $\overline D_{\pm}=D_\pm \cup \Gamma$.

\begin{definition}\label{defWHfact}\citep[Ch.1, Sec.1]{CG81}
Let $A(t)$  be a continuous  and non-singular $p\times p$ matrix
function on the contour $\Gamma $.

{A standard right Wiener--Hopf factorisation} of $A(t)$ is  its
representation in the form
$$
A(t)=A_-(t)D(t)A_+(t), \ t\in \Gamma.
$$
Here $A_{\pm}(t)$ are continuous and non-singular matrix functions on
$\Gamma $ that admit analytic continuation into  $D_{\pm}$ and their
continuations $A_{\pm}(z)$ are continuous and non-singular in $\overline D_{\pm}$;
$D(t)={\rm diag}\,[t^{\rho_1},\ldots,t^{\rho_p}]$, where integers
 $\rho_1,\ldots,\rho_p$ are called {the right partial indices}
of $A(t)$.

If $\rho_1=\cdots=\rho_p=0$, then the factorisation is called {canonical}.
\end{definition}
A similar definition is given for the standard left factorisation of $A(t)$
in which the factors $A_\pm(t)$ are reversed. In this work we consider only the standard right factorisation. The term standard factorisation refers to  a factorisation of a continuous matrix function $A(t)$ with continuous factors $A_{\pm}(t)$.

The partial indices $\rho_1,\ldots,\rho_p$ are uniquely determined (up to the ordering) by the matrix function $A(t)$ \citep[Theorem~1.1]{GF74}. The following important relations hold for the sums of the partial indices:
$\rho_1+\ldots+\rho_p={\rm ind}_\Gamma\,\det A(t)$.
Here ${\rm ind}_\Gamma\,\det A(t)$ is the winding number of $\det A(t)$ with respect to the contour
$\Gamma$.

The partial indices may be arranged in an arbitrary order. 
Any reordering of the indices reduces to replacing the factor $D(t)$ by the diagonal matrix function
$P D(t) P^{-1}$, where $P$ is a permutation matrix.
For this reason, in Definition \ref{defWHfact}
the partial indices  may be assumed to be ordered in any convenient way. 
Throughout this work, we adopt the following ordering  of the partial indices:
$\rho_1 \leqslant \ldots \leqslant \rho_p$. Note that the ordering of the indices affects 
a form of the normalised factors.

In general, a matrix function $A(t)$ with continuous entries does
not admit the standard  Wiener--Hopf factorisation. Hence, the class of matrix functions under consideration must be clarified. In what follows, we consider the standard factorisation in decomposing Banach algebras. A detailed exposition of the factorisation theory in such algebras is presented in \citep{CG81,GKS03}.

Let $\mathfrak A$ be an
algebra of continuous functions on $\Gamma$ such that any non-singular
matrix function $A(t)\in G\mathfrak A^{p\times p}$
admits the Wiener--Hopf factorisation with the factors
$A_{\pm}(t)\in G\mathfrak A_\pm^{p\times p}$. Here $G\mathfrak
A^{p\times p}$ is the group of invertible elements of the algebra
$p\times p$ matrix functions with entries in $\mathfrak A$.  
The requirement $A_\pm(t)\in G{\mathfrak A}_\pm^{p\times p}(\Gamma)$ means that 
the factors belong to the algebra ${\mathfrak A}(\Gamma)$, admit analytic extensions 
to the domains ${D}_\pm$ and are non-singular on 
${\overline {D}}_\pm$. Thus, the standard factorisation of a matrix function from a Banach algebra 
is a factorisation with factors from the same algebra.

Basic examples of such algebras are the Wiener algebra $W(\mathbb T)$ (or
more generally, a decomposing $R$-algebra) and the algebra
$H_\alpha(\Gamma)$ of H\"older continuous functions on $\Gamma $
\citep[Ch.2]{CG81}.

For $\mathfrak A= W(\mathbb T)$ or $H_\mu(\Gamma)$ the scalar
problem can be solved explicitly. In the matrix case explicit
formulas for the factors $A_{\pm}(t)$ and the partial indices of an
arbitrary matrix function have not yet been obtained. A review of available explicit and constructive factorisation methods can be found in \citep{RM16,KAMR21,GM13}.

In contrast to the partial indices,  the factors $A_\pm(t)$ are not  determined uniquely by 
$A(t)$. The general form of the factors $A_\pm(t)$ is described by the Gohberg--Krein theorem \citep[Ch. VIII, Theorem 1.2]{GF74}. The formulation of the theorem depends on the adopted ordering of the partial indices. We reformulate the theorem according the adopted ordering.

Let $\{k_1,\ldots,k_s\}$ be  a partition of $p\in\mathbb N$, that is 
$k_1+\cdots+k_s=p$. Let us represent a matrix $A\in {\mathbb C}^{p\times p}$ in 
the block form 
$$
A=\begin{pmatrix}
A_{11}&A_{12}&\cdots&A_{1s}\\
A_{21}&A_{22}&\cdots&A_{2s}\\
\vdots&\vdots&\ddots&\vdots\\
A_{s1}&A_{s2}&\cdots&A_{ss}\\
\end{pmatrix},
$$
where the block $A_{ij}$ has the size $k_i\times k_j$. We say that $A$ has the block structure generated by the partition $\{k_1,\ldots,k_s\}$ and denote it as 
$A\in {\mathcal B}(k_1,\ldots,k_s)$. If $A, B \in {\mathcal B}(k_1,\ldots,k_s)$, then  
we can perform multiplication of these matrices using block multiplication and 
$A\cdot B \in {\mathcal B}(k_1,\ldots,k_s)$.

Let $\rho_1,\ldots,\rho_p$ be the set of the right partial indices, sorted in the ascending order: $\rho_1 \leq \ldots \leq \rho_p$. Assume that this set contains $s$ different numbers $\varkappa_1<\ldots<\varkappa_s$ of multiplicity $k_1,\ldots, k_s$, respectively. 
Let $\mathcal Q_-(\rho_1,\ldots,\rho_p)$ be the set of all block triangular matrix functions of the form
\begin{equation}\label{gQminus}
Q_-(t)=\begin{pmatrix}
Q_{11}&Q_{12}&\ldots&Q_{1s}\\
0&Q_{22}&\ldots&Q_{2s}\\
\vdots&\vdots&\ddots&\vdots\\
0&0&\ldots&Q_{ss}
\end{pmatrix}.
\end{equation}
Here the block $Q_{ij}$ has dimensions $k_i\times k_j$, the diagonal blocks $Q_{ii}$ are constant invertible $k_i\times k_i$ matrices, and the off-diagonal blocks $Q_{ij}(t)$ are matrix polynomials in the variable $t^{-1}$ of degree at most $\varkappa_j-\varkappa_i$.  Thus, the block structure of $Q_-(t)$ is generated by the multiplicity 
$k_1,\ldots, k_s$ of the different partial indices $\varkappa_1<\ldots<\varkappa_s$.
For brevity, we call such a block structure a structure generated by the partial indices. The set $\mathcal Q_-(\rho_1,\ldots,\rho_p)$ is a subgroup of the group 
$G{\mathfrak A}_-^{p\times p}(\Gamma)$.

Let $D(t) ={\rm diag}\left[t^{\rho_1},\ldots,t^{\rho_p}\right]={\rm diag}\left[t^{\varkappa_1}I_{k_1},\ldots,t^{\varkappa_s}I_{k_s}\right]$.
Define a matrix function
$$Q_+(t)= D^{-1}(t)Q_-^{-1}(t)D(t).$$
$Q_+(t)$ has the same form as (\ref{gQminus}), only in this case, $Q_{ij}(t)$ are matrix polynomials in $t$ of degree at most $\varkappa_j-\varkappa_i $. Thus, 
$Q_+(z)\in G{\mathfrak A}_+^{p\times p}(\Gamma)$.

\begin{theorem}[Gohberg--Krein]\label{genformGK}
Let $A(t)\in G{\mathfrak A}^{p\times p}(\Gamma)$ and
\begin{equation}\label{factWH} 
A(t)=A_-(t)D(t)A_+(t)
\end{equation}
be a standard Wiener--Hopf factorisation of the matrix function $A(t)$ with partial indices $\rho_1 \leq \ldots \leq \rho_p$. 
Then the representation
\begin{equation}\label{generalFact}
A(t)=G_-(t)D(t)G_+(t),
\end{equation}
where $G_-(t)= A_-(t)Q_-(t)$, $G_+(t)= Q_+(t)A_+(t)$, is also a Wiener--Hopf factorisation of $A(t )$ for any $Q_-(t)\in \mathcal Q_-(\rho_1,\ldots,\rho_p)$. Moreover, any factorisation of $A(t)$ can be obtained from the original factorisation in a similar way with an appropriate choice of $Q_-(t)$.
\end{theorem}

\subsection{Birkhoff factorisation}
The next type of factorisation of matrix functions that we need was introduced by G. Birkhoff \citep{Birkhoff1913} in connection with some problems of ordinary differential equations. We present the basic facts about this factorisation according to monograph \citep{CG81}.

\begin{definition}\label{Birkfacto} 
\emph{A right Birkhoff factorisation} of a matrix function 
$A(t)\in G{\mathfrak A}^{p\times p}(\Gamma)$ is its representation in the following form:
\begin{equation}\label{Birkfact}
A(t)=D_b(t)B_-(t)B_+(t), \ \ t\in \Gamma,
\end{equation}
where $B_\pm(t) \in G{\mathfrak A}_\pm^{p\times p}(\Gamma)$ and 
$D_b(t) ={\rm diag}\left[t^{\beta_1},\ldots,t^{\beta_p}\right]$. The integers
$\beta_1, \ldots, \beta_p$ are called {the right Birkhoff indices} of $A(t)$.
\end{definition}

In contrast to the partial indices, the Birkhoff indices are
not  determined  uniquely by the matrix function $A(t)$ (see the corresponding examples 
in \citep[p.133]{CG81}, \citep[p.22]{GKS03}).

I. S. Chebotaru \citep{Chebotaru67} showed that if a Wiener--Hopf factorisation 
$$A(t) =A_-(t)D(t)A_+(t)$$ 
is known, then one can always construct a Birkhoff factorisation of the form
$$
A(t) = PD(t)P^{-1}B_-(t)B_+(t),
$$ 
where $P$ is a permutation matrix.
Thus, the existence of a Wiener–Hopf factorisation implies the existence of a
Birkhoff factorisation and among all possible sets of Birkhoff indices,
there is always one obtained by a permutation of the right partial indices of
the Wiener--Hopf factorisation. However, not every set of Birkhoff indices can be
obtained in this way. These results are discussed in detail in 
\citep[Ch. 6, Sec. 2]{CG81}.

 \subsection{LU and PLU factorisations of matrices}
The final type of factorisation we need is a block PLU factorisation of a
numerical matrix. We start with a simpler case that is well known in connection with solving the system $Ax=b$ by Gaussian elimination.

\begin{definition}\label{defLU}
Let $A\in {\mathbb C}^{p\times p}$ be a non-singular $p\times p$ matrix. 

{LU factorisation} of $A$ is its representation in the following form
\begin{equation}\label{LU}
A=L\cdot U,
\end{equation}
where $L$ is the lower triangular matrix, and $U$ is the upper triangular matrix.
\end{definition}

Not every non-singular matrix $A$ has  the LU factorisation.
$A$ admits  the LU factorisation if and only if all upper left $k\times k$ blocks of $A$,
$1\leq k \leq p$, are non-singular \citep[Corollary~3.5.5]{HJ86}. If we fix the diagonal elements of the matrix $L$, then the LU factorisation will be unique. We will assume that all diagonal elements of $L$ are equal to $1$. 

In the general case, by permutation of rows of $A$, it is always possible to obtain that all upper left blocks become non-singular. It means that
there exists a  permutation matrix $P^{-1}$ such that $P^{-1}A$ admits 
the LU factorisation, i.e. $A$ is represented as $A=PLU$ \citep[Theorem~3.5.7]{HJ86}. 
This is the {PLU factorisation}  of $A$. In general,  the permutation P is non-unique in this representation of $A$. In the similar way we can define the {PUL factorisation}.

\begin{remark}
LU and UL factorisations can be considered as the simplest examples of the canonical Wiener--Hopf factorisations in an abstract decomposing algebra.

Let ${\mathfrak A}={\mathbb C}^{p\times p}$ be the algebra of all complex $p\times p$
matrices. The subalgebras ${\mathfrak A}_+={\mathbb C}^{p\times p}_+$,
${\mathfrak A}_-={\mathbb C}^{p\times p}_-$, 
${\mathfrak A}_-^0={\mathbb C}^{0,p\times p}_-$
are defined as the algebras of all lower triangular matrices,  all the upper triangular matrices, and all the upper triangular matrices with the zero diagonal, respectively.
Then ${\mathfrak A}$ is the decomposing algebra: ${\mathfrak A}={\mathfrak A}_-^0\oplus{\mathfrak A}_+$.

If $A\in {\mathfrak A}$, then 
$$
A=LU=A_+A_-, \qquad A=\widetilde{U}\widetilde{L}=\widetilde{A}_-\widetilde{A}_+
$$
can be considered as the canonical left/right Wiener--Hopf factorisation in 
the decomposing algebra ${\mathfrak A}$.
\end{remark}

\subsection{Block LU and block PLU factorisations of matrices}
Let $\{k_1,\ldots,k_s\}$ be  a partition of $p\in\mathbb N$, that is 
$k_1+\cdots+k_s=p$. Let a $p\times p$ complex matrix $A$ have the block structure generated by this partition: $A\in {\mathcal B}(k_1,\ldots,k_s)$, i.e.  
the block $A_{ij}$ has the size $k_i\times k_j$. 

\begin{definition}\label{defblockLU}
Let $A\in {\mathbb C}^{p\times p}$ be a non-singular $p\times p$ matrix. The {block LU factorisation} of $A\in {\mathcal B}(k_1,\ldots,k_s)$ is its representation in the following form
\begin{equation}\label{blockLU}
A=L\cdot U,
\end{equation}
where $L, U \in {\mathcal B}(k_1,\ldots,k_s)$, and $L$ is the lower block triangular matrix,  $U$ is the block upper triangular matrix.
For the partition $\{1,\ldots,1\}$ we get the ordinary LU factorisation.
\end{definition}

Not every non-singular matrix $A$ has the block  LU factorisation.
A matrix $A$ admits the block LU factorisation if and only if all upper left sub-matrices of $A$ of sizes $k_1\times k_1, (k_1+k_2)\times (k_1+k_2),\ldots, (k_1+k_2+\cdots + k_s)\times 
(k_1+k_2+\cdots +k_s)$ are non-singular \citep[Item 11.14]{VK84}. If we fix the diagonal elements of the matrix $L$, then the block LU factorisation will be unique. We will assume that the diagonal elements of $L$ are equal to $I_{k_1},\ldots,I_{k_s}$. Here $I_k$ is an identity matrix of order $k$.

If  $A$ does not admit the block LU factorisation, then by permutation of rows of $A$, it is always possible to obtain that all corresponding submatrices of $A$ become non-singular. It means that there exists a  permutation matrix $P^{-1}$ such that $P^{-1}A$ admits the block LU factorisation, i.e. $A$ is represented as $A=PLU$. 
This is {the block PLU factorisation} of $A$.
 
Let us emphasise that $A$ admitting the block PLU factorisation with the given permutation matrix $P$ means that P is chosen in such a way that
all upper left sub-matrices of $P^{-1}A$ of sizes $k_1\times k_1, (k_1+k_2)\times (k_1+k_2),\ldots, (k_1+k_2+\cdots + k_s)\times (k_1+k_2+\cdots +k_s)$ are non-singular.

 \begin{remark}
During a normalising process of an original factorisation $A(t)=A_-(t)D(t)A_+(t)$
we will use the block PLU factorisation of the matrix $A_-(\infty)$. The matrix is endowed with the block structure generated by the  partial indices. If $A(t)$ has the distinct partial indices then the partition is $\{1,\ldots,1\}$ and we get the PLU factorisation of matrix $A_-(\infty)$. 
\end{remark}

\section{Definition of P-normalised factorisation}\label{sec2}
 First, let us clarify what we mean by normalisation at infinity. This normalisation is not assumed to yield a unique factorisation. Thus, it is an intermediate normalisation on the way to the final one, which guarantees uniqueness.

\begin{definition}\label{Nor}
The transition from the original factorisation~(\ref{factWH}) to the 
factorisation~(\ref{generalFact}) by using any matrix function 
$Q_-(t)\in \mathcal Q_-(\rho_1,\ldots,\rho_p) $ will be called {the  normalisation of the factorisation (\ref{factWH}) at infinity}. The matrix function $Q_-(t)$ is called the {normalisation matrix}.
\end{definition}

Thus, the normalisation of the factorisation at infinity is determined by the choice of 
$Q_-(t)\in \mathcal Q_-(\rho_1,\ldots,\rho_p)$. Since
$\mathcal Q_-(\rho_1,\ldots,\rho_p)$ is a group,  
we can perform the normalisation in several steps, choosing  at each step the required normalisation matrix.

Our task is to produce in some sense a canonical choice of $Q_-(t)$.
The main condition that determines the choice of a canonical representative 
$Q_-(t)\in \mathcal Q_-(\rho_1,\ldots,\rho_p)$
will be related to the Birkhoff factorisation of the matrix function $A(t)$.

\begin{definition}\label{canonPnorm}
Let $P$ be any  permutation matrix of order $p$.
The Wiener--Hopf factorisation of the matrix function $A(t)$:
\begin{equation}\label{rcanfact}
A(t)=C_-(t)D(t)C_+(t)
\end{equation}
is called {P-normalised} if the following conditions are fulfilled:
\begin{enumerate}
\item The matrix function $B_-(t)=PD^{-1}(t)P^{-1} C_-(t)D(t)$ belongs to the algebra 
$\mathfrak{A}_{-}^{p\times p}(\Gamma)$, 
\item $B_-(\infty)=P$.
\end{enumerate}

We will consider P-normalisation  as {canonical normalisation} of the factorisation.
\end{definition}

This definition does not specify how the permutation P should be chosen.
We clarify this question in the next section.

\section{Uniqueness of P-normalised factorisation}\label{sec3}
In this section, we prove that if a $P$-normalised factorisation exists, then it is unique.



\begin{theorem}\label{Bir}
Let $A(t)=A_{-}(t)D(t)A_{+}(t)$ be a known factorisation of a matrix function $A(t)$. For the matrix $A_-(\infty)$, fix a permutation $P$ such that $A_-(\infty)$ admits a block PLU factorisation. 

Let the matrix function $A(t)\in G{\mathfrak A}^{p\times p}(\Gamma)$ admit a $P$-normalised Wiener--Hopf factorisation:
$$
A(t)=C_-(t)D(t)C_+(t).
$$
Then
\begin{enumerate}
\item This Wiener--Hopf factorisation generates a Birkhoff factorisation by the formula
\begin{equation}\label{Bir_form}
A(t)=PD(t)P^{-1}B_-(t)B_+(t),
\end{equation}
where $B_-(t)=PD^{-1}(t)P^{-1} C_-(t)D(t)$, $B_+(t)=C_+(t)$.

\item The given $P$-normalised Wiener--Hopf factorisation is unique.
\end{enumerate}
\end{theorem}
\begin{proof} \, By the first condition of Definition~ \ref{canonPnorm}, we have $B_-(t)\in GW_{-}^{p\times p}(\mathbb{T})$. Then relation \eqref{Bir_form} is verified directly.

Let us prove the uniqueness of the $P$-normalised Wiener--Hopf factorisation.
Assume that $A(t)=\widetilde C_-(t)D(t)\widetilde C_+(t)$ is another $P$-normalised factorisation of $A(t)$ and
$$A(t)=PD(t)P^{-1} \widetilde B_-(t)\widetilde B_+(t),$$ is the corresponding Birkhoff factorisation.
Then $\widetilde B_-^{-1}(t)B_-(t)= \widetilde B_+(t)B_+^{-1}(t)$, and, therefore, by Liouville's theorem, this matrix function is a constant invertible matrix. Hence, $\widetilde B_-^{-1}(t)B_-(t)=\widetilde B_-^{-1}(\infty)B_-(\infty)=I_p$, due to the second condition of Definition \ref{canonPnorm}. Thus, $\widetilde B_-(t)=B_-(t)$ and $\widetilde C_-(t)=C_-(t)$, $\widetilde C_+(t)=C_+(t)$ .
\end{proof}

\begin{remark}
The condition $B_-(\infty)=P$, which ensures the uniqueness of the $P$-normalised factorisation, can be replaced by $B_-(\infty)= A_0$, where $A_0$ is any invertible matrix. The condition $B_-(\infty)=P$ allows us to obtain a simplest form of the factors $C_-(t)$, $B_-(t)$ in $P$-normalised factorisations.
\end{remark}

We now prove that any Wiener--Hopf factorisation can be used to determine the permutation matrix $P$.



\begin{theorem}\label{sness}
Let a matrix function $A(t)$ admit 
the $P$-normalised factorisation for some permutation matrix $P$. Then, for any Wiener--Hopf factorisation
$$A(t)=A_-(t)D(t)A_+(t),$$
the matrix $A_-(\infty)$ admits a block PLU factorisation. Here block structure of
$A_-(\infty)$ is generated by the right partial indices of $A(t)$.
\end{theorem}
\begin{proof} \, Suppose that the matrix function $A(t)$ admits a $P$-normalised factorisation
$$
A(t)=C_-(t)D(t)C_+(t),
$$
where $P$ is some permutation matrix. First we prove that for this factorisation $C_-(\infty)$ admits a block PLU factorisation, where the block structure is generated by the right partial indices of $A(t)$. 

Let $\varkappa_1<\ldots<\varkappa_s$ be the distinct partial indices of $A(t)$ and 
$k_1,\ldots, k_s$ their multiplicities. Thus
$$
D(t)=\begin{pmatrix}
t^{\varkappa_1}I_{k_1}&0&\cdots&0\\
0&t^{\varkappa_2}I_{k_2}&\cdots&0\\
\vdots&\vdots&\ddots&\vdots\\
0&0&\cdots&t^{\varkappa_s}I_{k_s}
\end{pmatrix}.
$$

By the first condition of Definition~\ref{canonPnorm},
the following inclusion holds:
\begin{equation}\label{condi1}
PD^{-1}(t)\bigl(P^{-1} C_-(t)\bigr)D(t)\in \mathfrak{A}_{-}^{p\times p}(\Gamma).
\end{equation}
Denote
$$
L_-(t)=P^{-1} C_-(t)\in \mathfrak{A}_{-}^{p\times p}(\Gamma).
$$
Define $\widetilde{L}_-(t):= D^{-1}(t)L_-(t)D(t)$. Then $\widetilde{L}_-(t) \in \mathfrak{A}_{-}^{p\times p}(\Gamma)$ by virtue (\ref{condi1}) and 
$${L}_-(t)= D(t)\widetilde{L}_-(t)D(t).
$$

Partition $\widetilde {L}_-(t)$ into blocks: 
$$
\widetilde {L}_-(t)=\begin{pmatrix}
L_{11}^-(t)&L_{12}^-(t)&\cdots&L_{1s}^-(t)\\
L_{21}^-(t)&L_{22}^-(t)&\cdots&L_{2s}^-(t)\\
\vdots&\vdots&\ddots&\vdots\\
L_{s1}^-(t)&L_{s2}^-(t)&\cdots&L_{ss}^-(t)
\end{pmatrix},
$$
where the block $L_{ij}^-(t)\in {\mathfrak A}_-^{k_i \times k_j}$.

Then we have
$$
L_-(t)=D(t){\widetilde L}_-(t)D^{-1}(t)=
\begin{pmatrix}
L_{11}^-(t)&t^{-\varkappa_{21}}L_{12}^-(t)&\cdots&t^{-\varkappa_{s1}}L_{1s}^-(t)\\
t^{\varkappa_{21}}L_{21}^-(t)&L_{22}^-(t)&\cdots&t^{-\varkappa_{s2}}L_{2s}^-(t)\\
\vdots&\vdots&\ddots&\vdots\\
t^{\varkappa_{s1}}L_{s1}^-(t)&t^{\varkappa_{s2}}L_{s2}^-(t)&\cdots&L_{ss}^-(t)
\end{pmatrix}.
$$
Here $\varkappa_{ji}=\varkappa_j-\varkappa_i>0$ for $j>i$.

The condition $L_-(t)\in \mathfrak{A}_{-}^{p\times p}(\Gamma)$
implies that $t^{\varkappa_{ij}}L_{ij}^-(t):={\widetilde L}_{ij}^-(t) \in {\mathfrak A}_-^{k_i\times k_j}$ for $i>j$.
Consequently, 
$$
L_0:=L_-(\infty)=\begin{pmatrix}
L_{11}^-(\infty)&0&\cdots&0\\
{\widetilde L}_{21}^-(\infty)&{L}_{22}^-(\infty)&\cdots&0\\
\vdots&\vdots&\ddots&\vdots\\
{\widetilde L}_{s1}^-(\infty)&{\widetilde L}_{s2}^-(\infty)&\cdots&L_{ss}^-(\infty)
\end{pmatrix},
$$
is block lower triangular matrix. Hence,
$$
C_-(\infty)=PL_0,
$$
that is, for the $P$-normalised factorisation $$
A(t)=C_-(t)D(t)C_+(t),
$$
the matrix $C_-(\infty)$ admits a block PLU factorisation.

Further, by Theorem~\ref{genformGK},
for any Wiener--Hopf factorisation
$$
A(t)=A_-(t)D(t)A_+(t)
$$
the following representation holds:
$$
A_-(t)=C_-(t)Q_-(t),
$$
where $Q_-(t)$ has the form~(\ref{gQminus}). In particular,
$$
A_-(\infty)=C_-(\infty)Q_-(\infty).
$$
Denote $U_0:=Q_-(\infty)$. Then $U_0$ is an upper
block triangular matrix.

Thus,
$$
A_-(\infty)=P\bigl(P^{-1}C_-(\infty)\bigr)Q_-(\infty)
=PL_0U_0,
$$
that is, the matrix $A_-(\infty)$ admits a block PLU factorisation.
\end{proof}

\section{Existence of P-normalised factorisation}\label{sec4}





The following theorem gives the main result of the paper. Let 
$\mathfrak A$ be an algebra of continuous functions on $\Gamma$ such that any matrix function $A(t)\in G\mathfrak A^{p\times p}$ admits the Wiener--Hopf factorisation with the factors $A_{\pm}(t)\in G\mathfrak A_\pm^{p\times p}$.
Recall that a block structure of matrices is understood as a structure generated by 
the right partial indices $\varrho_1, \dots, \varrho_p$ of $A(t)$.

\begin{theorem}\label{basethm}  
Let $A(t)\in G\mathfrak A^{p\times p}$.  $A(t)$ admits a $P$-normalised right Wiener--Hopf factorisation if and only if for some factorisation
$A(t) = A_-(t) D(t) A_+(t)$
the matrix $A_-(\infty)$ admits the block PLU factorisation for the given permutation 
$P$.

If this condition is fulfilled then the P-normalised
factorisation 
\begin{equation}\label{canonnormfact}
A(t) = C_-(t) D(t)C_+(t) 
\end{equation}
has the form
\begin{equation}\label{canonnormAm}
C_-(t) = P
\begin{pmatrix}
I_{k_1} + t^{-1}c_{11} 
& t^{-\varkappa_{21}-1} c_{12}  
& \cdots 
& t^{-\varkappa_{s-1,1}-1} c_{1,s-1}
& t^{-\varkappa_{s1}-1} c_{1s} \\
c_{21} 
& I_{k_2} + t^{-1}c_{22}  
& \cdots 
& t^{-\varkappa_{s-1,2}-1} c_{2,s-1}
& t^{-\varkappa_{s2}-1} c_{2s} \\
\vdots & \vdots & \ddots & \vdots & \vdots\\
c_{s-1,1} 
&c _{s-1,2} 
& \cdots 
& I_{k_{s-1}} + t^{-1}c_{s-1,s-1}\quad
& t^{-\varkappa_{s,s-1}-1} c_{s-1, s} 
\\
c_{s1} 
& c_{s2} 
& \cdots 
&c _{s,s-1}
& I_{k_s} + t^{-1}c_{ss}
\end{pmatrix},
\end{equation}
$C_+(t)= Q_+(t)A_+(t)$.
Here $c_{ij}(t) \in {\mathfrak A}_{-}^{k_i \times k_j}$. We will call this form of the factor $C_-(t)$ canonical.
\end{theorem}
\begin{proof} \, The necessity of the  condition of the theorem is proved in 
Theorem~\ref{sness}.

Now we prove the sufficiency. First, we prove the existence of $P$-normalised factorisation for $P = I$. This means that the following conditions are satisfied:

\begin{center} 
\emph{all upper left sub-matrices of $A_-(\infty)$ of sizes}
$
k_1\times k_1, (k_1+k_2)\times (k_1+k_2),\ldots, (k_1+k_2+\cdots + k_s)\times (k_1+k_2+\cdots +k_s)  
$
\emph{are non-singular.}
\end{center}
 
We prove the sufficiency by induction on the number $k$ of upper left blocks of $A_-^{(k)}(t)$ in canonical form in the factorisation $A(t) = A_-^{(k)}(t) D(t) A_+^{(k)}(t)$. 

For the case $P = I$ the matrix $A_-(\infty):= A_0$ admits the block LU factorisation:
$
A_0 = L \cdot U,
$
where $L$ is a lower block triangular matrix with $I_{k_1}, \dots, I_{k_s}$ as diagonal elements and $U$ is an upper block triangular matrix. This block LU factorisation is unique.

\noindent\textbf{1st step of the normalising process.}
Let us take $Q_1 = U^{-1}$ as a constant normalising matrix.
Obviously, $Q_1 \in Q_-(\varrho_1, \dots, \varrho_p)$, that is, $Q_1$ is a normalising matrix. Then
\begin{equation}\label{1step}
A(t) = A_-^{(1)}(t) D(t) A_+^{(1)}(t), \quad t\in \Gamma,
\end{equation}
where $A_-^{(1)}(t) = A_-(t) Q_1$, $A_+^{(1)}(t) = D^{-1}(t)Q_1^{-1}D(t)A_+(t)$,
is a new factorisation of $A(t)$. This factorisation is normalised by the condition 
$
A_-^{(1)}(\infty) = A_-(\infty) U^{-1} = L,
$
that is
$$
A_-^{(1)}(t) = \left(
\begin{array}{c|ccc} 
I_{k_1} + t^{-1} A_{11}^{(1)} & t^{-1} A_{12}^{(1)} & \dots & t^{-1} A_{1s}^{(1)} \\
\hline
A_{21}^{(1)} & I_{k_2} + t^{-1} A_{22}^{(1)} & \dots & t^{-1} A_{2s}^{(1)} \\
\vdots & \vdots & \ddots & \vdots \\
A_{s1}^{(1)} & A_{s2}^{(1)} & \dots & I_{k_s} + t^{-1} A_{ss}^{(1)}
\end{array}
\right), \quad A_{ij}^{(1)}(t) \in {\mathfrak A}_{-}^{k_i \times k_j}.
$$

This means that after the first step we obtain the factorisation for which the upper
left block of  size $k_1\times k_1$ 
has the canonical form. For certain matrix functions the process can be finished after this
step.

\noindent\textbf{2nd step of the normalising process.} Now we normalise the 
factorisation (\ref{1step}) so that the upper left block of size 
$(k_1+k_2)\times (k_1+k_2)$ has the canonical form.
Select the upper left block of size $(k_1+k_2)\times (k_1+k_2)$ in the matrix function
$A_-^{(1)}(t)$:
$$
G_{11}^{(1)}(t)=\begin{pmatrix} 
I_{k_1} + t^{-1} A_{11}^{(1)} & t^{-1} A_{12}^{(1)}\\
A_{21}^{(1)} & I_{k_2} + t^{-1} A_{22}^{(1)} \\
\end{pmatrix}
$$
and represent $A_-^{(1)}(t)$ in the block form
\begin{equation}\label{blockA1}
A_-^{(1)}(t)=\left(\begin{array}{c|c}
G_{11}^{(1)}&G_{12}^{(1)}\\
\hline
G_{21}^{(1)}&G_{22}^{(1)}
\end{array}
\right).
\end{equation}

Put 
$
Q_2(t)= \left(\begin{array}{c|c}
Q_{11}^{(2)}&0\\
\hline
0&I_{k_3+\cdots +k_s}
\end{array}
\right),
$
where
$
Q_{11}^{(2)}(t)=\begin{pmatrix}
I_{k_1}&t^{-1}q_{12}^{(2)}\\
0&I_{k_2}
\end{pmatrix}
$
and $q_{12}^{(2)}(t)$ is $k_1\times k_2$ matrix polynomial in $t^{-1}$ of degree at most
$\varkappa_{21}-1$. The coefficients of $q_{12}^{(2)}(t)$ should be determined.
$Q_2(t)$ has the same block structure as $A_-^{(1)}(t)$ in (\ref{blockA1}).
Obviously, $Q_2(t)$ belongs to the group $Q_-(\varrho_1, \dots, \varrho_p)$.

Let us set $A_-^{(2)}(t) = A_-^{(1)}(t) Q_2(t)$, $A_+^{(2)}(t) = D^{-1}(t)Q_2^{-1}(t)D(t)A_+^{(1)}(t)$.
Then
\begin{equation}\label{2step}
A(t) = A_-^{(2)}(t) D(t) A_+^{(2)}(t), \quad t\in \Gamma,
\end{equation}
is the factorisation of $A(t)$. Here
$
A_-^{(2)}=\left(\begin{array}{c|c}
G_{11}^{(1)}Q_{11}^{(2)}&G_{12}^{(1)}\\
\hline
G_{21}^{(1)}Q_{11}^{(2)}&G_{22}^{(1)}
\end{array}
\right).
$
It is not difficult to see that
$$
G_{21}^{(1)}Q_{11}^{(2)}=\begin{pmatrix}
A_{31}^{(2)}&A_{32}^{(2)}\\
\vdots&\vdots\\
A_{s1}^{(2)}&A_{s2}^{(2)}
\end{pmatrix}.
$$
Let us find the block
$$
G_{11}^{(1)}Q_{11}^{(2)}=\begin{pmatrix} 
I_{k_1} + t^{-1} A_{11}^{(1)} \quad
&t^{-1}\left[(I_{k_1} + 
t^{-1} A_{11}^{(1)})q_{12}^{(2)}+A_{12}^{(1)}\right]\\
A_{21}^{(1)} & I_{k_2} + t^{-1} A_{22}^{(1)} \\
\end{pmatrix}.
$$
The matrix function $I_{k_1} + t^{-1} A_{11}^{(1)}(t)$ is invertible in a neighborhood of infinity and we have
$$
t^{-1}\left[(I_{k_1} + 
t^{-1} A_{11}^{(1)})q_{12}^{(2)}+A_{12}^{(1)}\right]=t^{-1}(I_{k_1} + 
t^{-1} A_{11}^{(1)})\left[q_{12}^{(2)}+
(I_{k_1} + t^{-1} A_{11}^{(1)})^{-1}A_{12}^{(1)}\right].
$$
The matrix function $\left(I_{k_1} + t^{-1} A_{11}^{(1)}(t)\right)^{-1}A_{12}^{(1)}(t)$ is analytic
in a neighborhood of infinity. Expand it into the Laurent series and put 
$q_{12}^{(2)}(t)=-P_{\varkappa_{21}-1}(t)$. Here $P_{\varkappa_{21}-1}(t)$
is the corresponding Laurent polynomial for this series. Then matrix function
$
t^{-1}(I_{k_1} + t^{-1} A_{11}^{(1)})\left[q_{12}^{(2)}+(I_{k_1} + t^{-1} A_{11}^{(1)})^{-1}A_{12}^{(1)}\right]
$ 
can be represented as
$t^{-\varkappa_{21}-1}A_{12}^{(2)}(t)$, 
where $A_{12}^{(2)}(t)\in {\mathfrak A}_-^{k_1\times k_2}$.

Thus, we get 
$$
A_-^{(2)}(t) = \left(
\begin{array}{cc|ccc} 
I_{k_1} + t^{-1} A_{11}^{(2)} & t^{-\varkappa_{21}-1} A_{12}^{(2)} 
& t^{-1} A_{13}^{(2)}& \dots & t^{-1} A_{1s}^{(2)} \\
A_{21}^{(2)} & I_{k_2} + t^{-1} A_{22}^{(2)} &t^{-1}A_{23}^{(2)}&  \dots & t^{-1} A_{2s}^{(2)} \\
\hline
A_{31}^{(2)} & A_{32}^{(2)} & I_{k_3} + t^{-1} A_{33}^{(2)} & \dots & t^{-1} A_{2s}^{(2)} \\
\vdots & \vdots & \vdots&\ddots & \vdots \\
A_{s1}^{(2)} & A_{s2}^{(2)} & A_{s3}^{(2)} &\dots & I_{k_s} + t^{-1} A_{ss}^{(2)}
\end{array}
\right),
$$
where all $A_{ij}^{(2)}(t) \in {\mathfrak A}_{-}^{k_i \times k_j}$.

Therefore, using the matrix $Q_-^{(2)}(t)=Q_2(t)Q_1\in Q_-(\varrho_1, \dots, \varrho_p)$ as a normalising matrix, we normalise the factorisation (\ref{factWH}) so that the upper left block of size 
$(k_1+k_2)\times (k_1+k_2)$ of $A_-^{(2)}(t)$ has the canonical form.
For certain matrix functions, the process terminates after this
step.

\noindent\textbf{(m+1)-th step of a normalising process.}
Suppose there exists a normalising matrix 
$Q_-^{(m)}(t) \in {\mathcal Q}_{-}(\varrho_1, \dots, \varrho_p)$, for some $m \ge 2$, that reduces the factorisation (\ref{factWH}) to the factorisation

\begin{equation}\label{mfact}
A(t) = A_-^{(m)}(t) D(t) A_+^{(m)}(t), \quad t\in \Gamma,
\end{equation}
where $A_-^{(m)}(t)$ has the upper left block of size
$(k_1 + \dots + k_m) \times (k_1 + \dots + k_m)$ in the canonical form.

Let us select two blocks in $A_-^{(m)}(t)$: the $(k_1 + \dots + k_m) \times (k_1 + \dots + k_m)$ canonical block, and the next diagonal block that we must reduce to the canonical form.
For brevity, we denote the diagonal elements by 
$D_{jj}^{(m)}(t):=I_{k_j}+t^{-1}A_{jj}^{(m)}(t)$.
Now we have
\begin{multline*}
A_-^{(m)}(t)=\\ 
\left(
\begin{array}{cccc|c|ccc}
D_{11}^{(m)}
& t^{-\varkappa_{21}-1}A_{12}^{(m)}
& \cdots
& t^{-\varkappa_{m 1}-1}A_{1m}^{(m)}
& t^{-1}A_{1,m+1}^{(m)}
& t^{-1}A_{1,m+2}^{(m)}
& \cdots
& t^{-1}A_{1,s}^{(m)}
\\
A_{21}^{(m)}
& D_{22}^{(m)}
& \cdots
& t^{-\varkappa_{m 2}-1}A_{2m}^{(m)}
& t^{-1}A_{2,m+1}^{(m)}
& t^{-1}A_{2,m+2}^{(m)}
& \cdots
& t^{-1}A_{2,s}^{(m)}
\\
\vdots & \vdots & \ddots & \vdots & \vdots  & \vdots&&\vdots
\\
A_{m 1}^{(m)}
& A_{m 2}^{(m)}
& \cdots
& D_{mm}^{(m)}
& t^{-1}A_{m,m+1}^{(m)}
& t^{-1}A_{m,m+2}^{(m)}
& \cdots
& t^{-1}A_{m, s}^{(m)}
\\
\hline
A_{m+1,1}^{(m)}
& A_{m+1,2}^{(m)}
& \cdots
& A_{m+1,m}^{(m)}
& D_{m+1,m+1}^{(m)}
& t^{-1}A_{m+1,m+2}^{(m)}
& \cdots
& t^{-1}A_{m+1,s}^{(m)}
\\
\hline
A_{m+2,1}^{(m)}
& A_{m+2,2}^{(m)}
& \cdots
& A_{m+2,m}^{(m)}
&A_{m+2,m+1}^{(m)}
& D_{m+2,m+2}^{(m)}
&\cdots
& t^{-1}A_{m+2,s}^{(m)}
\\
\vdots & \vdots & \ddots & \vdots & \vdots &\vdots &\ddots & \vdots
\\
A_{s, 1}^{(m)}
& A_{s, 2}^{(m)}
& \cdots
& A_{s,m}^{(m)}
& A_{s,m+1}^{(m)}
& A_{s,m+2}^{(m)}
& \cdots
& D_{ss}^{(m)}
\end{array}
\right)
\end{multline*}

Thus, we obtain the following block form of $A_-^{(m)}(t)$:
\begin{equation}\label{mFactWH}
A_-^{(m)}(t)=\left(\begin{array}{c|c|c}
G_{11}^{(m)}&G_{12}^{(m)}&G_{13}^{(m)}\\
\hline
G_{21}^{(m)}&G_{22}^{(m)}&G_{23}^{(m)}\\
\hline
G_{31}^{(m)}&G_{32}^{(m)}&G_{33}^{(m)}
\end{array}
\right).
\end{equation}
Now we  form a normalising matrix $Q_{m+1}(t)\in {\mathcal Q}_{-}(\varrho_1, \dots, \varrho_p)$ with the same block structure as $A_-^{(m)}(t)$ in (\ref{mFactWH}):
$$
Q_{m+1}(t) =
\left(\begin{array}{c|c|c}
I_{k_1+\ldots+k_m}&t^{-1}q_{m+1}&0\\
\hline
0&I_{k_{m+1}}
&0\\
\hline
0&0&I_{k_{m+2} + \dots + k_s}
\end{array}
\right), {\mbox\ {\rm where}\ \ }
q_{m+1}(t) =
\begin{pmatrix}
q_{1,m+1}^{(m+1)} \\
\vdots \\
q_{m,m+1}^{(m+1)} \\
\end{pmatrix},
$$
and $q_{ij}^{(m+1)}(t)$ is $k_i\times k_j$ matrix polynomial in $t^{-1}$ of degree at most
$\varkappa_{ji}-1$.

Let us normalise the factorisation (\ref{mfact}) by $Q_{m+1}(t) \in Q_{-}(\varrho_1, \dots, \varrho_p)$:
\begin{equation}\label{mplus1fact}
A(t) = A_-^{(m+1)}(t) D(t) A_+^{(m+1)}(t), \quad t\in \Gamma.
\end{equation}
Here $A_-^{(m+1)}(t)=A_-^{(m)}(t)Q_{m+1}(t), \quad 
A_+^{(m+1)}(t) = D^{-1}(t)Q_{m+1}^{-1}(t)D(t)A_+^{(m)}(t).
$

Then for $A_-^{(m+1)}(t)=A_-^{(m)}(t)Q_{m+1}(t)$ we have:
$$
A_-^{(m+1)}(t)=
\left(
\begin{array}{c|c|c}
G_{11}^{(m)}&t^{-1}G_{11}^{(m)}q_{m+1}+G_{12}^{(m)}&G_{13}^{(m)}\\
\hline
G_{21}^{(m)}&t^{-1}G_{21}^{(m)}q_{m+1}+G_{22}^{(m)}&G_{23}^{(m)}\\
\hline
G_ {31}^{(m)}&t^{-1}G_{31}^{(m)}q_{m+1}+G_{32}^{(m)}&G_{33}^{(m)}\\
\end{array}
\right),
$$
where the block
$$
t^{-1}G_{21}^{(m)}q_{m+1}+G_{22}^{(m)}=
t^{-1}\left(
A_{m+1,1}^{(m)}, \cdots, A_{m+1,m}^{(m)}\right)
\begin{pmatrix}
q_{1,m+1}^{(m+1)} \\
\vdots \\
q_{m,m+1}^{(m+1)} \\
\end{pmatrix}
+D_{m+1,m+1}^{(m)}$$ has the  form $D_{m+1,m+1}^{(m+1)}$,
and
$
t^{-1}G_{31}^{(m)}q_{m+1}+G_{32}^{(m)}=
\begin{pmatrix}
A_{m+2,m+1}^{(m+1)}\\
\vdots\\
A_{s,m+1}^{(m+1)}
\end{pmatrix}.
$

It  remains to find the block $t^{-1}G_{11}^{(m)}(t)q_{m+1}(t)+G_{12}^{(m)}(t)$ in $A_-^{(m+1)}$.

Let us write the block $G_{12}^{(m)}(t)$ in the form:
$G_{12}^{(m)}(t)=t^{-1}A_{m+1}(t)$, where
$A_{m+1}(t)=\begin{pmatrix}
A_{1,m+1}^{(m)}\\
\vdots\\
A_{m,m+1}^{(m)}\\
\end{pmatrix}.
$
Now 
$$
t^{-1}G_{11}^{(m)}(t)q_{m+1}(t)+G_{12}^{(m)}(t)= t^{-1}\left[G_{11}^{(m)}(t)q_{m+1}(t)+
A_{m+1}(t)\right].
$$
Since the matrix $G_{11}^{(m)}(\infty)=
\begin{pmatrix}
I_{k_1}&0&\dots&0\\
*&I_{k_2}&\dots&0\\
\vdots&\vdots&\ddots&\vdots\\
*&*&\dots&I_{k_m}
\end{pmatrix}
$ is invertible, the matrix function $G_{11}^{(m)}(t)$ is invertible in a neighborhood of infinity
and
$$t^{-1}G_{11}^{(m)}(t)q_{m+1}(t)+G_{12}^{(m)}(t)= t^{-1}G_{11}^{(m)}(t)
\left[q_{m+1}(t)+(G_{11}^{(m)}(t))^{-1}A_{m+1}(t)\right].$$
Here the block column $(G_{11}^{(m)}(t))^{-1}A_{m+1}(t)$ is analytic at infinity.
We expand all the block elements of this column into Laurent series in a neighborhood of infinity and represent $(G_{11}^{(m)}(t))^{-1}A_{m+1}(t)$ in the following form:
$$
(G_{11}^{(m)}(t))^{-1}A_{m+1}(t)=
\begin{pmatrix}
P_{\varkappa_{m+1,1}-1}(t)+t^{-\varkappa_{m+1,1}}O_{1,m+1}(t)\\
\vdots\\
P_{\varkappa_{m+1,m}-1}(t)+t^{-\varkappa_{m+1,m}}O_{m,m+1}(t)\\
\end{pmatrix}.
$$
Here $P_{\varkappa_{m+1,j}-1}(t)$ is the corresponding Laurent polynomial
of the series and $O_{j,m+1}(t)\in {\mathfrak A}^{k_j\times k_{m+1}}$.
Now we set $q_{j,m+1}^{(m+1)}(t)= -P_{\varkappa_{m+1,j}-1}(t)$ and arrive 
at the following equality:
$$
t^{-1}G_{11}^{(m)}(t)q_{m+1}(t)+G_{12}^{(m)}(t)=
\begin{pmatrix}
t^{-\varkappa_{m+1,1}-1}A_{1,m+1}^{m+1}(t)\\
\vdots\\
t^{-\varkappa_{m+1,m}-1}A_{m,m+1}^{m+1}(t)
\end{pmatrix}:=t^{-1}G_{11}^{(m)}(t)
\begin{pmatrix}
t^{-\varkappa_{m+1,1}}O_{1,m+1}(t)\\
\vdots\\
t^{-\varkappa_{m+1,m}}O_{m,m+1}(t)\\
\end{pmatrix}.
$$

Therefore, the normalising matrix $Q_-^{(m)}(t)Q_{m+1}(t)\in Q_-(\varrho_1, \dots, \varrho_p)$ reduces the factorisation (\ref{factWH}) to the factorisation (\ref{mplus1fact}) in which the upper left block of size 
$(k_1+\cdots +k_{m+1})\times (k_1+\cdots +k_{m+1})$ of $A_-^{(m+1)}(t)$ has the canonical form. For the case $P=I_p$, by induction we get the factorisation (\ref{canonnormfact}) with the factor $A_-(t)$ in the canonical form (\ref{canonnormAm}).
It remains to  verify that this factorisation is an I-normalised factorisation.

Direct calculation shows that 
\begin{multline*}
B_-(t)=D^{-1}(t)C_-(t)D = \\
\begin{pmatrix}
I_{k_1} + t^{-1}c_{11} 
& t^{-1} c_{12}  
& \cdots 
& t^{-1} c_{1,s-1}
& t^{-1} c_{1s} \\
t^{-\varkappa_{21}}c_{21} 
& I_{k_2} + t^{-1}c_{22}  
& \cdots 
& t^{-1} c_{2,s-1}
& t^{-1} c_{2s} \\
\vdots & \vdots & \ddots & \vdots & \vdots\\
t^{-\varkappa_{s-1,1}}c_{s-1,1} 
&\ t^{-\varkappa_{s-1,2}}c _{s-1,2} 
& \cdots 
& I_{k_{s-1}} + t^{-1}c_{s-1,s-1}\quad
& t^{-1} c_{s-1, s} 
\\
t^{-\varkappa_{s,1}}c_{s1} 
& t^{-\varkappa_{s,2}}c_{s2} 
& \cdots 
&t^{-\varkappa_{s,s-1,1}}c _{s,s-1}
& I_{k_s} + t^{-1}c_{ss}
\end{pmatrix},
\end{multline*}
and conditions (1), (2) of Definition \ref{canonPnorm} are fulfilled. 
Hence the factorisation with the factor $C_-(t)$ in the canonical form is indeed an 
I-normalised factorisation.

Now we consider the general case. Suppose that for  $A(t)$ there exists a factorisation $$A(t)=A_-(t)D(t)A_+(t),$$ 
where $A_-(\infty)$ admits a block PLU factorisation with some permutation $P$:
$A_-(\infty)=PLU$.
Note that in this case any factorisation of $A(t)$ has this property.

Let us define the matrix function $F(t)=P^{-1}A(t)$. Then 
$$F(t)=(P^{-1}A_-(t))D(t)A_+(t)$$ 
is the right Wiener--Hopf factorisation of $F(t)$, and for $F_-(t)=P^ {-1}A_-(t)$ the condition $F_-(\infty)=LU$ is satisfied.
By the part of the theorem already proved, $F(t)$ admits the $I$-normalised Wiener--Hopf factorisation $F(t)=K_-(t)D(t)F_+(t)$. Then $A(t)=(PK_-(t))D(t)F_+(t)$ is the $P$-normalised factorisation of $A(t)$. The theorem is completely proved.
\end{proof}

\begin{example} 
Let $A(t)= \begin{pmatrix}
36t^2+17t-14&t^2&\quad t+1\\
t+1&t^2&\quad t+1\\
t^4-t^2+3t-1&\quad t^2+13t+15&\quad t\\

\end{pmatrix}$. Then
\begin{multline*}
A(t)=A_-(t)D(t)A_+(t)=
\begin{pmatrix}
1&\quad0&-1\\
1&\quad 0&\quad \frac{4}{9t}-\frac{5}{12t^2}\\
0&\quad 1&0\\
\end{pmatrix}\cdot 
\begin{pmatrix}
1&\quad 0&\quad 0\\
0&\quad 1&\quad 0\\
0&\quad 0&\quad t^2
\end{pmatrix}\cdot\\
\begin{pmatrix}
17t-14&\quad t^2&\quad t+1\\
t^4-t^2+3t-1&\quad t^2+13t+15&\quad t\\
-36&0&\quad 0
\end{pmatrix}
\end{multline*}
is the right Wiener--Hopf factorisation of $A(t)$. This factorisation was obtained using our package ExactMPF \citep{AAM22} that is designed for the construction of the exact factorisation of matrix polynomials. From the factorisation we see that $\varkappa_1=0$, 
$k_1=2$; $\varkappa_2=2$, $k_2=1$. 

Consider the matrix $A_-(\infty)$ with the block structure 
$\{k_1,k_2\}=\{2, 1\}$:
$$
A_-(\infty)=\left(\begin{array}{cc|c}
1&0&-1\\
1&0&0\\
\hline
0&1&0
\end{array}\right).
$$
For the permutation $P=P^{-1}=\begin{pmatrix}
1&0&0\\
0&0&1\\
0&1&0
\end{pmatrix}$ we have
$
P^{-1}A_-(\infty)=\left(\begin{array}{cc|c}
1&0&-1\\
0&1&0\\
\hline
1&0&0
\end{array}\right).
$
Hence $A_-(\infty)$ admits the block PLU factorisation: 
$$A_-(\infty)=P\cdot L \cdot U=P\cdot
\left(\begin{array}{cc|c}
1&0&0\\
0&1&0\\
\hline
1&0&1
\end{array}\right)\cdot
\left(\begin{array}{cc|c}
1&0&-1\\
0&1&0\\
\hline
0&0&1
\end{array}\right),
$$ 
and the  matrix function $A(t)$ admits the $P$-normalised factorisation.
Define the matrix function $F(t)=P^{-1}A(t)$. Then $F(t)$ admits the factorisation
\begin{multline*}
F(t)= F_-(t)D(t)F_+(t)=
\begin{pmatrix}
1&\quad 0&-1\\
0&\quad 1&0\\
1&\quad 0&\quad \frac{4}{9t}-\frac{5}{12t^2}\\
\end{pmatrix}\cdot 
\begin{pmatrix}
1&\quad 0&\quad 0\\
0&\quad 1&\quad 0\\
0&\quad 0&\quad t^2
\end{pmatrix}\cdot\\
\begin{pmatrix}
17t-14&\quad t^2&\quad t+1\\
t^4-t^2+3t-1&\quad t^2+13t+15&\quad t\\
-36&0&\quad 0
\end{pmatrix}.
\end{multline*}

Our goal is to I-normalise this factorisation.

\noindent\textbf{1st step of the normalising process.}
Let us take $Q_1 = U^{-1}= \left(\begin{array}{cc|c}
1&0&1\\
0&1&0\\
\hline
0&0&1
\end{array}\right)$.

Then $F(t)=F_-^{(1)}(t)D(t)F_+^{(1)}(t)$, where
$$
F_-^{(1)}(t)=F_-(t)Q_1=\begin{pmatrix}
1&\quad 0&0\\
0&\quad 1&0\\
1&\quad 0&\quad 1+\frac{4}{9t}-\frac{5}{12t^2}\\
\end{pmatrix},
$$
$$
F_+^{(1)}(t)=D^{-1}(t)Q_1D(t)F_+(t)=\begin{pmatrix}
36t^2+17t-14&t^2&t+1\\
t^4-t^2+3t-1&t^2+13t+15&t\\
-36&0&9   
\end{pmatrix}.
$$

We see that $F_-^{(1)}(t)$ has the canonical form and the normalising process for $F(t)$terminates at the first step.

Now $A(t)=(PF_-^{(1)}(t))D(t)F_+^{(1)}(t)$ is a $P$-normalised factorisation of $A(t)$. 
\end{example}

\section{Conclusions and perspectives}\label{sec5}

In this article, we proposed the definition of $P$-normalised Wiener--Hopf factorisation and demonstrated that this normalisation guarantees the uniqueness of the factorisation. We presented a finite theoretical normalising process for constructing this factorisation, thereby filling the gap that previously existed in the general theory of factorisation. The primary tools used to achieve this result were the Gohberg--Krein theorem on the general form of factorisation factors, the Birkhoff factorisation of matrix functions, and the block PLU factorisation of numerical matrices. However, in this paper, we did not address the numerical aspects of our normalisation process. 

Regarding this problem, we make the following remarks.
  
\begin{enumerate}

\item To initiate the normalising process, an initial factorisation $A(t)=A_-(t)D(t)A_+(t)$ of the matrix function $A(t)$ must be available. However, there are no explicit methods for finding the Wiener--Hopf factorisation for arbitrary matrix functions, which remains a central problem in factorisation theory. Nevertheless, there are several classes of matrix functions for which the problem of explicit construction has been solved, and the normalising procedure can be applied to them. These include triangular matrix functions \citep{Adukov82, PR18}, as well as meromorphic and analytic matrix functions \citep{Adukov92, Adukov99}. In particular, explicit formulas for the partial indices and factorisation factors $A_\pm(t)$ exist for matrix polynomials. This class is crucial for the approximate construction of the factorisation. Currently, the \texttt{ExactMPF} package allows for the exact solution of the factorisation problem for matrix polynomials \citep{AAM22}. Our first priority is to incorporate the $P$-normalisation option into this package. 

\item Knowing $A_-(\infty)$ and the partial indices, we can find a permutation $P$ for which the matrix $P^{-1}A_-(\infty)$ has non-singular upper left sub-matrices of sizes
$$
k_1\times k_1, (k_1+k_2)\times (k_1+k_2),\ldots, (k_1+k_2+\cdots + k_s)\times (k_1+k_2+\cdots +k_s).  
$$
These conditions can be achieved by rearranging the rows of the matrix $A_-(\infty)$. It is clear that the permutation $P$ is not unique. In particular, we can choose it such that all upper corner minors are non-zero. This permutation determines the type of normalisation.

\item Next, we must construct the block LU factorisation of $P^{-1}A_-(\infty)$. Algorithms for such factorisations are well developed in numerical linear algebra \citep[Ch. 4]{TB97}. Moreover, modern computer algebra systems allow one to obtain $P$, $L$, and $U$ easily.

\item Subsequently, we can sequentially perform the finite sequence  of steps described in the proof of Theorem~\ref{basethm} to construct the $P$-normalised factorisation.

\item It is straightforward to prove that the type of $P$-normalisation is stable under small perturbations of the original matrix function $A(t)$. This means that for sufficiently small $\delta>0$, any matrix function $\widetilde A(t)$ satisfying the following conditions:
\begin{itemize}
\item $\widetilde A(t)$ has the same set of right partial indices $\rho_1,\ldots,\rho_p$ as $A(t)$;

\item $\|A(t)-\widetilde A(t)\|<\delta$,
\end{itemize}

has the same $P$-type normalisation as $A(t)$. However, we currently lack an effective estimate for the value of $\delta$. Without such an estimate, the normalisation cannot be utilized in an approximate factorisation process. Our second priority is to derive such an effective estimate.

\item Furthermore, knowing the $P$-normalised factorisation of $A(t)$, we will be able to obtain effective estimates for $\|A_\pm-{\widetilde A}_\pm\|$. This constitutes our third task. By solving it, we will be able to obtain an approximate factorisation with guaranteed accuracy for certain classes of matrix functions. 
\end{enumerate}  

We intend to address these problems in our future work.


\section*{Funding}
This work is supported by  the Russian Science Foundation,  project \#25-21-20134. 


\bibliographystyle{abbrvnat}
\bibliography{AdukovsRefs_New}

\end{document}